\documentclass{amsart}

\usepackage{amsfonts}
\usepackage{amsmath}
\usepackage{amssymb}
\usepackage{url}
\usepackage{hyperref}
 \usepackage{color}
 \hypersetup{	
colorlinks=true, 
breaklinks=true, 
urlcolor= blue, 
linkcolor= red,	
citecolor= blue,	
} 
\usepackage{lipsum}
\usepackage[title]{appendix}
\usepackage{mathtools}
\usepackage{ulem}
\usepackage[utf8]{inputenc}
\usepackage{cleveref}
\usepackage[paperheight=10.75in,paperwidth=8.25in,margin=1in,heightrounded]{geometry}

\usepackage{graphicx}

\def\R{\mathbb{R}}

\def\H{\mathbb{H}}

\def\E{\mathbb E}

\newtheorem{theorem}{{Theorem}}[section]
\newtheorem*{theorem*}{{Theorem}}
\newtheorem*{corollary*}{{Corollary}}
\newtheorem{proposition}[theorem]{{Proposition}}
\newtheorem{isom.ext}[theorem]{{Trivial isometric extension}}
\newtheorem{lemma}[theorem]{{Lemma}}
\newtheorem{corollary}[theorem]{{Corollary}}
\newtheorem{remark}[theorem]{{Remark}}

 \title[Flexibility and rigidity of conformal embeddings in Lorentzian manifolds]{Flexibility and rigidity of conformal embeddings in Lorentzian manifolds
 }

\author[Alaa Boukholkhal]{Alaa Boukholkhal}\thanks{The author is supported by the LABEX MILYON (ANR-10-LABX-0070) of
Université de Lyon, within the program "Investissements d'Avenir" (ANR-11-IDEX-0007) operated by the French National Research Agency (ANR)}
\address{UMPA, École normale supérieure de Lyon, France}
\email{mohamed.boukholkhal@ens-lyon.fr}
\begin{document}

\begin{center}
\begin{abstract} 
We prove that for any Riemannian metric $g$ on a closed orientable surface $\Sigma$ and any spacelike embedding $f:\Sigma \rightarrow M$ in a pseudo-Riemannian manifold $(M,h)$, the embedding $f$ can be $C^{0}$-approximated by a smooth conformal embedding for $g$. If in addition, $M$ is a quotient of the $(2+1)$-dimensional solid timelike cone by a cocompact lattice of $SO^{\circ}(2,1)$, we show that the set of negatively curved metrics on $\Sigma$ that admit isometric embeddings in $M$ projects into a relatively compact set in the Teichm\"uller space.
\end{abstract}
\end{center}

\maketitle

\section{Introduction}
 Conformal structures on a surface are in one to one correspondence with complex structures, also called Riemann surface structures. It is therefore natural to ask whether we can find for every Riemann surface a conformal model embedded in the Euclidean space $\E^{3}=(\R^{3},g_{euc})$. In other terms, for any closed orientable Riemannian surface $(\Sigma,g)$, is there a smooth embedding $f:\Sigma \rightarrow \E^{3}$, such that $f^{*}g_{euc}$ is conformal to $g$ (i.e. $f^{*}g_{euc} = \lambda g$ where $\lambda$ is a positive function on $\Sigma$)?

If we allow embeddings of regularity $C^{1}$, we can find isometric, and thus conformal, embeddings by the Nash-Kuiper theorem \cite{nash1954c1},\cite{kuiper1955c1}. For higher regularity, isometric embeddings do not always exist. For example, any $C^{2}$-embedded closed surface in $\E^{3}$ has a point with positive curvature. Therefore, there is no $C^{2}$-isometric embedding of flat tori or closed hyperbolic surfaces in $\E^{3}$.

All conformal structures on the sphere are equivalent by the uniformization theorem, and hence any embedding of the sphere is conformal. It is also known that every rectangular elliptic curve is conformal to a torus of revolution in $\E^{3}$. The first example of a conformally embedded non-rectangular torus was given by Teichm\"uller \cite{zbMATH03097094}. In the 1960s, Garsia proved that every closed Riemann surface can be smoothly conformally embedded in $\E^{3}$ \cite{zbMATH03215435},\cite{garsia1960imbedding}. In 1971, R\"uedy generalized Garsia's result to open Riemann surfaces \cite{ruedy1971embeddings}. These results were further generalized by Ko, who proved that any Riemann surface can be conformally embedded in any orientable Riemannian manifold of dimension $\geq 3$ \cite{ko1989embedding}, \cite{zbMATH01542084}. In this paper, we will consider the conformal embedding problem when the ambient space is pseudo-Riemannian. One obvious obstruction here is the existence of a spacelike embedding i.e. an embedding for which the induced metric is Riemannian. For instance, there are no spacelike embeddings of closed surfaces in the $(2+1)$-dimensional Minkowski space. We 
 will prove that this is the only obstruction, more precisely we have the following: 

\begin{theorem}\label{T1}
Let $(M,h)$ be a pseudo-Riemannian manifold of dimension $\ge 3$ and $f: \Sigma \rightarrow (M,h) $ a spacelike embedding of a closed orientable surface $\Sigma$. For any Riemannian metric $g$ on $\Sigma$, there exists a smooth spacelike embedding $F : \Sigma \rightarrow M$ such that:
    \begin{itemize}
        \item $F$ is conformal for the metric $g$. 
        \item $F$ is $C^{0}$-close to $f$.
    \end{itemize}

\end{theorem}

\begin{remark}
    The $C^{0}$-closeness in the theorem is stated for any choice of a Riemannian metric on $M$. 
\end{remark}

\medskip
One important case where this theorem can be applied is for quotients of the open solid timelike cone $\mathcal{T}$ of dimension $2+1$ .i.e
$$\mathcal{T}= \{ (x,y,z)\in \mathbb{R}^{2,1} |  \; x^{2}+y^{2} < z^{2} \; ; \; z> 0 \}$$ 
Indeed, for any lattice $\Gamma$ of $SO^{\circ}(2,1)$ isomorphic to $\pi_1(\Sigma)$, where $\Sigma$ is a closed surface of genus $\geq 2$, we can consider the quotient $\mathcal{T}/\Gamma$. We have a natural isometric embedding of $\Sigma$ with a hyperbolic metric $h_\Gamma$ realized as the quotient of the hyperboloid $\H^{2}$ by $\Gamma$, in particular, $\Sigma$ admits a spacelike embedding in  $\mathcal{T}/\Gamma$. Applying Theorem \ref{T1} in this context, we get: 
\begin{corollary}
    Let $\Sigma$ be a closed surface of genus $\geq 2$, and let $\Gamma\leq SO^{\circ}(2,1)$ a lattice isomorphic to $\pi_1(\Sigma)$. Then, any Riemannian metric on $\Sigma$ admits a smooth conformal embedding in $\mathcal{T}/\Gamma$.
\end{corollary}
While smooth conformal embeddings exhibit high flexibility in this setting, we will prove that smooth isometric embeddings remain rigid. Indeed, an immediate application of the maximum principle implies that $h_\Gamma$ is the only hyperbolic metric that can be smoothly embedded in $\mathcal{T}/\Gamma$. Moreover, in \cite{lab}, Labourie and Schlenker proved that for any negatively curved metric on $\Sigma$, there exists a unique quotient of $\mathcal{T}$, where this metric can be isometrically embedded (see \cite{fillastre}, \cite{Graham} for other variants of this result).

In view of this rigidity, it is natural to ask: for a fixed $\Gamma$, which negatively curved metrics can be isometrically embedded in $\mathcal{T}/\Gamma$? Such an embedding exists if and only if there is an equivariant isometric embedding of the universal cover of $\Sigma$ into $\mathcal{T}$. In particular, this embedding will be the graph of a convex function. The manifold $\mathcal{T}/\Gamma$ is a flat maximal globally hyperbolic spatially compact spacetime. These spacetimes have been classified by Mess in \cite{mess2007lorentz}, and the study of convex Cauchy surfaces in this context is the subject of many papers (see for instance \cite{zeghib}, \cite{seppi}).

  Although any conformal structure on $\Sigma$ has a model in $\mathcal{T}/\Gamma$ by Theorem \ref{T1}, we will show that this model can not be always negatively curved. More precisely, denote by:
\begin{itemize}
    \item $\mathcal{M}_{<0}$ the set of Riemannian metrics on $\Sigma$ of negative curvature.
    \item $\widetilde{B}(\Sigma)$ the set of metrics in $\mathcal{M}_{<0}$ that can be smoothly isometrically embedded in $\mathcal{T}/\Gamma$.
    \item $B(\Sigma)$ the projection of $\widetilde{B}(\Sigma)$ in the Teichm\"uller space $Teich(\Sigma)$ (i.e. space of Riemannian metrics on $\Sigma$ up to conformal scaling and up to diffeomorphisms isotopic to the identity).
\end{itemize}
\begin{theorem}\label{T2}
   Let $\mathcal{T}$ and $\Gamma$ be as defined above. Then there exists a constant $C > 1$ that depends only on $\Gamma$, such that for any metric $g\in \widetilde{B}(\Sigma)$, up to scaling $g$ by a constant factor, we have:  
    $$\frac{1}{C} h_\Gamma \leq g \leq C h_\Gamma.$$
\end{theorem}
This translates in the Teichm\"uller space to the following fact:
\begin{corollary}\label{C2}
    The set $B(\Sigma)$ is relatively compact. In particular, there are some conformal structures on $\Sigma$ that can  not be realized by convex embeddings in $\mathcal{T}/\Gamma$.
\end{corollary}
\begin{remark}
  This is the only flat maximal globally hyperbolic spatially compact spacetime with such a property. Indeed, If $M$ is a $(2+1)$-flat maximal globally hyperbolic spatially compact spacetime (with a hyperbolic Cauchy surface) and non isometric to $\mathcal{T}/\Gamma$, then by the work of Barbot, Béguin and Zeghib \cite{zeghib}, $M$ is foliated by spacelike copies of $\Sigma$ with constant negative curvature that projects into a curve converging to the boundary in the Teichm\"uller space (see \cite{Belraouti}).
\end{remark}
 
\subsection{Structure of the paper} 
\begin{itemize}
    \item In section $2$, we present the construction of almost isometric maps that will be used in the proof of Theorem \ref{T1}. The construction uses Gromov's convex integration theory \cite{gromov1986}.
    \item In section $3$, we recall some elements of Teichm\"uller theory.
    \item In section $4$, we give the proof of Theorem \ref{T1} for genus $1$ surfaces and then for higher genus. 
    \item The last section is devoted to the proof of Theorem \ref{T2}. The problem is presented in an equivariant form. We establish that any convex embedding is bounded by two hyperbolic levels that depend only on $\Gamma$. We then, prove that the tangent planes are uniformly spacelike.
\end{itemize}   
\subsection{Acknowledgement}
This work is part of my PhD research. I am grateful to my PhD supervisors, Aurélien Alvarez and Vincent Borrelli, for their encouragement and support throughout this work, and for their helpful remarks on the early drafts of this paper. I am also thankful to Ghani Zeghib for suggesting the rigidity problem and for the many helpful discussions about it. 
\section{Almost isometric maps}
 
In the following, we present the main construction used to prove Theorem \ref{T1}. Here and in the rest of the paper, we will consider as target only Lorentzian manifolds of dimension $3$, but all the constructions and arguments extend directly to the general case.
\medskip

Consider $f: \Sigma \rightarrow M$ to be a spacelike embedding of $\Sigma$ ($f^*h$ defines a Riemannian metric on $\Sigma$), and let $g$ be a Riemannian metric on $\Sigma$. We will use a variant of the convex integration formula \cite{spring}, called the corrugation process \cite{theilliere2022convex}. This process applied to $f$ will allow us to find $\varepsilon$-isometric embeddings for the metric $g$, where by an $\varepsilon$-isometric embedding, we mean an embedding $F: \Sigma \rightarrow (M, h)$ such that $\|F^{*}h - g\|_{C^{0}} \leq \varepsilon$.

\subsection{Local construction}
We start with the local construction as the corrugation process formula is defined locally. To illustrate the idea, we present the result in the most simple setting:

\begin{proposition}\label{local}
    Let $(M, h)$ be the Minkowski space $\mathbb{R}^{2,1}$ with the usual Lorentzian metric and let $f : C := [0,1]^2 \rightarrow M$ be a spacelike embedding\footnote{We care only about what happens in the interior of $C$.}. Consider the metric $\mu = f^{*}h - \eta \, d\varpi \otimes d\varpi$, where $\varpi : C \rightarrow \mathbb{R}$ is an affine projection and $\eta : C \rightarrow \mathbb{R}_{\geq 0}$ is a smooth function. Then, for any $\varepsilon > 0$, there exists an $\varepsilon$-isometric spacelike embedding $F : (C, \mu) \rightarrow (M, h)$, i.e. an embedding such that
    $$
    \|F^{*}h - \mu\|_{C^{0}} \leq \varepsilon.
    $$
\end{proposition}

\begin{proof}
    
 The corrugation process formula in this setting is given by: 
$$
\forall p \in C, \quad F(p)=f(p)+\frac{1}{N} \int_{0}^{N \varpi(p)}(\gamma(p, s)-\bar{\gamma}(p)) \mathrm{d} s
$$
where $\gamma: C \times(\mathbb{R} / \mathbb{Z}) \rightarrow \mathbb{R}^{2,1}$ is a smooth family of loops that will be defined below, $N$ is an integer called the corrugation number, and $\Bar{\gamma}(x) = \int_0^1 \gamma(x, t) \, dt$ is the average of $\gamma$ at each point.

Let $(v,u)$ be an orthonormal basis for the metric $f^*h$ on $C$, such that $v\in ker(d\varpi)$, and let 
$$\mathbf{v}=df(v) \; \; \; \; \mathbf{u}=df(u)$$
Consider $\mathbf{n}(p)$ to be the normal to $T_{f(p)}f(C)$ with respect to the metric $h$. Note that $\mathbf{n}$ is a timelike vector field of norm $-1$ since $f$ is spacelike.
We can now define the loop family $\gamma$ by: 

\begin{equation}\label{loop}
    \gamma(\cdot, s):=r(\cosh (\theta) \mathbf{u}+ \sinh (\theta) \mathbf{n}) \quad \text {with} \quad \theta=\alpha \cos (2 \pi s)
\end{equation}
and where $r$ and $\alpha$ are smooth functions that will be chosen below such that the average $\bar{\gamma}$ satisfy: 

$$\bar{\gamma} =\frac{\mathbf{u}}{d\varpi(u)}$$
Differentiating $F$, we get: 
$$
    dF  = df + (\gamma(p, N \varpi(p))-\bar{\gamma}(p)) \otimes d \varpi+\frac{1}{N} \int_{0}^{N \varpi(p)} d(\gamma(p, s)-\bar{\gamma}(p)) \mathrm{d} s .$$
Let $L$ be the target differential defined by: $$
L:=d f+(\gamma(\cdot, N \varpi)-\bar{\gamma}) \otimes d \varpi
$$
 This target differential coincides with $d f$ on ker $d \varpi$ and modifies it on the transversal directions by the addition of a term depending on $\gamma$. 
\\It can be checked that  $ \int_{0}^{N \varpi(p)}(\gamma(p, s)-\bar{\gamma}(p)) \mathrm{d}s= \int_{\lfloor N\varpi(p) \rfloor}^{N \varpi(p)}(\gamma(p, s)-\bar{\gamma}(p)) \mathrm{d}s $ and since $\gamma$ is continuous on a compact set, we will get 
$$F^*h = L^*h + O(\frac{1}{N}) $$
Of course if $N$ is big enough, then $F$ induces a Riemannian metric on $C$ if $L$ does. We will now show that by a convenient choice of $r$ and $\alpha$, we can get $L^{*}h=\mu$. Recall that:
$$ L= df + (r(\cosh (\theta) \mathbf{u}+ \sinh (\theta) \mathbf{n}) - \frac{\mathbf{u}}{d\varpi(u)}) \otimes d\varpi $$
and hence 
$$L^*h = f^*h + (r^{2} - \frac{1}{d\varpi(u)^{2}})d\varpi\otimes d\varpi$$
We choose now $r$ so that $$r^{2} = \frac{1}{d\varpi(u)^{2}} - \eta $$
this is of course possible since we assumed $\mu$ to be Riemannian. This implies that $L$ is $\mu$-isometric ($L^{*}h=\mu$).
\\ The only thing left now is to choose $\alpha$, so that 
$$\bar{\gamma} =\frac{\mathbf{u}}{d\varpi(u)}$$
Recall that the average of $\gamma$ is given by
$$\bar{\gamma}  =r\left(\int_{0}^{1} \cosh (\alpha \cos (2 \pi s)) \mathrm{d} s\right) \mathbf{u}$$
therefore, $\alpha$ must satisfy 
$$\int_{0}^{1} \cosh (\alpha \cos (2 \pi s)) \mathrm{d} s=\frac{1}{rd\varpi(u)}  $$
Since the function $$\alpha \in [0,+\infty[ \mapsto \left(\int_{0}^{1} \cosh (\alpha \cos (2 \pi s)) \mathrm{d} s\right) \in [1,+\infty[$$ is smooth, bijective and $$ \frac{1}{rd\varpi(u)} \ge 1$$ 
the above equation has a unique solution that varies smoothly.
Choosing $N$ big enough, we conclude that $F$ is $\varepsilon$-isometric for the metric $\mu$.
\end{proof}
   \begin{remark}\label{gluing rk}
    Notice that for any $x\in C$ such that $\eta(x)=0$ ($\mu(x)=f^{*}h (x)$), we have $f(x)=F(x)$. 
\end{remark}

Proposition \ref{local} generalizes immediately for general Lorentzian manifolds. It suffices to consider the general formula of the corrugation process given by: 

    Starting from a map $f_0 : U \rightarrow (M, h)$ on an open set $U \subset \Sigma$, a submersion $\varpi : U \rightarrow \mathbb{R}$, and a smooth loop family $\gamma : U \times \mathbb{R}/\mathbb{Z} \rightarrow f^*TM$ such that $\gamma(x, \cdot) : \mathbb{R}/\mathbb{Z} \rightarrow f^*T_xM$ for every $x \in U$. The map defined by the Corrugation Process is:
    $$
    F = CP_\gamma (f, \varpi, N) : x \mapsto \exp_{f(x)} \left( \frac{1}{N} \Gamma(x, N\varpi(x)) \right)
    $$
    where $\Gamma(x, s) = \int_0^s (\gamma(x, t) - \Bar{\gamma}(x)) \, dt$ and $\Bar{\gamma}(x) = \int_0^1 \gamma(x, t) \, dt$.

 If we take $U$ to be compact and such that $f^*TM$ is trivial, then we can find a smaller neighborhood of the zero section of $f^*TM \rightarrow U$ where the exponential map is well defined. Since $\Gamma(x,s)$ is bounded, we can choose $N$ big enough so that $\frac{1}{N}\Gamma(x,N\varpi(x))$ lies inside this neighborhood. With the same loop family \ref{loop} as before we get $$F^*h=f^*h - \eta d\varpi^{2} + O(\frac{1}{N}).$$

\subsection{Global construction}
    Even though the corrugation process is defined locally, we can still use it to construct global almost isometric maps.
    \begin{proposition}
        Let $(\Sigma,g)$ be a closed Riemannian surface and $(M,h)$ a Lorentzian manifold. For any long spacelike embedding $f:\Sigma \rightarrow M$ (i.e, $\Delta =f^*h-g $ is semi-definite positive), and any $\varepsilon>0$, there exists a spacelike embedding $F : \Sigma \rightarrow M$ such that :
$$\left\|F^{*}h-g\right\|_{C^{0}}\leq \varepsilon$$ 
    \end{proposition}

The construction follows the approach of Nash \cite{nash1954c1} and it proceeds as follows: 
\begin{itemize}
\item We choose a finite set $\phi_i: U_i \subset \mathbb{R}^2 \rightarrow V_i \subset \Sigma $ of local parametrizations of $\Sigma$ and a partition of unity $\{\rho_i\}$ subordinated to it ($supp(\rho_i) \subset U_i$) for all $i\in [\![1,K]\!]$.
    \item We identify the space of inner products of $\mathbb{R}^{2}$ (given by symmetric positive definite matrices $\begin{bmatrix} E & F \\ F & G \\ \end{bmatrix}$) with the interior of the cone $S^{+}(\mathbb{R}^{2}) \subset \R^{3}$ defined by the conditions:
    $$
\begin{cases} EG-F^{2} > 0, \\ E, G > 0 \end{cases}
$$
\item We note $\Delta_i =:\phi_i^*\Delta_{|\phi_i(supp(\rho_i))}$ so that we have
$$ \Delta = \sum_i \rho_i \Delta_i.$$
Since $supp(\rho_i)$ is compact and $\Delta_i : supp(\rho_i) \rightarrow S^{+}(\mathbb{R}^{2})$ is continuous, we can find $p_i$ squares of linear forms $\ell_{i,1}\otimes \ell_{i,1} , \ell_{i,2}\otimes \ell_{i,2}, ..., \ell_{i,p_i}\otimes \ell_{i,p_i}$ such that 
$$\Delta_i = \sum_{j=1}^{{p_i}} \eta_{i,j} \ell_{i,j} \otimes \ell_{i,j}$$ 
where each $\eta_{i,j} : supp(\rho_i) \rightarrow \mathbb{R}^{+} $ is a smooth function. For a proof of this fact, check lemma $29.3.1$ in \cite{eliash}.
\item Working chart by chart, we will build a sequence of intermediary maps $f=F_0, F_1, ..., F_K$ defined on $\Sigma$ such that
$$\left\|F_i^{*}h-(f^{*}h - \sum_{k=1}^{i}\rho_j \Delta_j)\right\|_{C^{0}}\leq \varepsilon$$
The map $F_K$ will be $\varepsilon$-isometric for the metric $g$ provided the corrugation numbers are large enough in each step.
 
\item Each map $F_i$ is built by applying $p_i$ successive corrugation process on $F_{i-1}$. We will use proposition \ref{local} in the chart $U_i$ to build a sequence of maps $F_{i-1}=F_i^{0}, F_i^{1},..., F_i^{p_i}=F_i $, where 
$$F_i^{j}=CP(F_i^{j-1}, \mu_i^{j}, \ell_{i,j}, N_{i,j})$$ 
and $\mu_i^{j} = F_{i}^{j-1 *}h - \rho_i\eta_{i,j} \ell_{i,j} \otimes \ell_{i,j}$. By remark \ref{gluing rk}, the maps are well defined on $\Sigma$. In fact, they all coincide outside $\phi_i(supp(\rho_i))$ since the metric $\mu_i^{j}$ is equal to $F_{i-1}^*h$ there.
\end{itemize}

\section{Some Teichm\"uller theory }
In the previous section, we managed to build for any Riemannian metric $g$ an embedding that is arbitrarily close to be conformal for $g$. This already proves that the set of Riemannian metrics on $\Sigma$ that can be conformally embedded in $(M,h)$ is dense in the space of metrics. Before proceeding into the proof of the embedding Theorem \ref{T1}, we will recall some facts from Teichm\"uller theory (for more details, check \cite{imayoshi2012introduction}, \cite{hubbard2016teichmuller}).

 Let $\Sigma$ be a closed orientable surface of genus $d\geq 1$. The Teichm\"uller space $Teich(\Sigma)$ can be described as the set of marked Riemann surfaces (pairs $(X,\varphi_X)$ consisting of a Riemann surface $X$ and a homeomorphism $\varphi_X:\Sigma\rightarrow X$) where we identify any two marked Riemann surfaces $(X,\varphi_X)$ and $(Y,\varphi_Y)$ if there exists a biholomorphism in the homotopy class of $\varphi_X\circ \varphi_Y^{-1}$. For simplicity, we will note by $(X,\varphi_X)$ for the class of $(X,\varphi_X)$ in $Teich(\Sigma)$. 
 
       Let $X$ and $Y$ be two closed Riemann surfaces of genus $d\geq 1$ and let $\varphi$ be an orientation preserving homeomorphism between them. We will consider for our purpose only maps that are smooth outside of a finite number of point $E \subset \Sigma$. We want to compare the complex structure of $X$ and $Y$.

  We consider in local coordinates, the dilatation of $\varphi$ to be the quantity
 $$Dil(\varphi)=\sup_{z\in \Sigma \setminus E}\left(\sup \left|\frac{d\varphi}{dz}\right|\bigg/ \inf \left|\frac{d\varphi}{dz}\right|\right)$$ 

 We say that the map $\varphi$ is quasiconformal if and only if $Dil(\varphi) < \infty$. Notice that $Dil(\varphi) \geq 1$ and $\varphi$ is conformal if and only if $Dil(\varphi)=1$.

 A distance on $Teich(\Sigma)$ called the Teichm\"uller distance can be defined by
  $$
    d_{Teich}((X,\varphi),(Y,\phi))= \frac{1}{2}\inf \log(Dil(\widetilde{\varphi}\circ \widetilde{\phi}^{-1})),$$
where the infimum is taken over all quasiconformal maps $\widetilde{\varphi}$ and $\widetilde{\phi}$ such that $\widetilde{\varphi}\circ \widetilde{\phi}^{-1}$ is homotopic to $\varphi\circ \phi^{-1}$.

Suppose now that we have a surface \(\Sigma\) and two Riemannian metrics \(g_1\) and \(g_2\) on it. We denote the associated complex structures by \(X(g_1)\) and \(X(g_2)\), respectively. We will need in the following to estimate the distance between $X(g_1)$ and $X(g_2)$. For that, we need to compute the dilatation $Dil(id)$ of the identity map from \(X(g_1)\) to \(X(g_2)\), given by
$$
    Dil(id) = \sup_{z\in \Sigma}\sqrt{\sup \frac{g_2}{g_1} \bigg/ \inf \frac{g_2}{g_1}},
$$
where the supremum and infimum are taken over all directions. In particular, we have
\begin{equation}\label{dila}
   d\left((X(g_1), id_\Sigma), (X(g_2), id_\Sigma)\right) \leq \frac{1}{2} \log(Dil(id)). 
\end{equation}

  \section{Proof of Theorem \ref{T1}}
Let $\Sigma$ be a closed orientable surface of genus $d \geq 1$, $(M,h)$ a pseudo-Riemannian manifold of $\dim \geq 3$ and $f: \Sigma \rightarrow M$ a spacelike embedding. For a fixed Riemann surface structure $X$ on $\Sigma$, we want to find a conformal embedding in $(M,h)$.  The strategy is as follows:
\begin{itemize}
\item Consider $B_\rho$ to be a closed ball in $Teich(\Sigma)$ centered at $(X,id_\Sigma)$ and of radius $\rho$. 
    \item We use the corrugation process to build a family of spacelike embeddings $F_q : \Sigma \rightarrow M$ for each $q\in B_\rho$.
    \item For each $q\in B_\rho$, the embedding $F_q$ defines a Riemann surface structure $X_q$ on $\Sigma$ induced by the metric $F_q^{*}h$. We denote by $G(q)$ the point $(X_q,id_\Sigma)$ in $Teich(\Sigma)$.
    \item We prove that the map $G$ is continuous, and it satisfies that $\|G(q)-q\|\leq \rho$ for every $q\in B_\rho$.
    \item Using a variant of Brouwer's fixed point theorem (proposition \ref{B} below), we find that $(X,id)\in Im(G)$. Therefore, there exists $q\in B_\rho$ such that $F_q$ is a smooth conformal embedding for $X$.
   
\end{itemize} 
 
 We distinguish two cases: when $\Sigma$ has genus one, and then higher genus case. But before that, let us prove the last ingredient of the proof:
 \begin{proposition}\label{B}
Let $\psi:\overline{B(p_0,r)} \rightarrow \mathbb{R}^n$ be a continuous map from the closed Euclidean ball of center $p_0 \in \mathbb{R}^n$ and radius $r>0$ that satisfies
$$ \parallel \psi(p)-p \parallel \leq r,$$
then $p_0 \in Im(\psi)$.
\end{proposition}
\begin{proof}
Applying Brouwer's fixed point Theorem to the map $p\longmapsto p_0+p-\psi(p)$, implies that there exists $p_1\in\overline{B(p_0,r)}$ such that $p_0+p_1-\psi(p_1)=p_1$, and hence, $\psi(p_1)=p_0$.
\end{proof} 

\subsection{Genus one case}
Let $\Sigma$ be a closed surface of genus one and let $f : \Sigma\longrightarrow  M$ be a spacelike embedding of $\Sigma$ into $M$. The Lorentzian metric $h$ induces a metric on $\Sigma$ and hence a conformal structure. For simplicity, we will assume that $\Sigma$ with this conformal structure is conformally equivalent to $\mathbb{C}/(\mathbb{Z}+ i\mathbb{Z})$. Hence, we see $f$ as a map from $C=[0,1]^{2}$ to $M$. We can also assume that the pullback of the metric has the form :
$$ f^{*}h= \lambda(z)^2|dz|^2 $$ 
\\ By the uniformization theorem, we know that any elliptic curve is conformally equivalent to the quotient of the complex plane by some lattice, i.e. $\mathbb{C}/\Lambda$ where $\Lambda = \langle\tau_1, \tau_2\rangle$
for some pair of translations $\tau_1$ and $\tau_2$. Now, since rotations and homotheties are biholomorphisms isotopic to identity, we can transform any pair $\tau_1$ and $\tau_2$ so that $\tau_1 = 1$ and $\tau_2$ is in the upper half plane, i.e. $\tau_2\in \mathbb{H}^2=\{z\in\mathbb{C}\; | \;Im(z) > 0 \}$.

From now on, we identify the Teichm\"uller space of $\Sigma$ with $\mathbb{H}^{2}$. We have for $w\in \mathbb{H}^{2}$ 
$$ \varphi_w :z\in \mathbb{C} \longrightarrow \frac{i+w}{2}(z+\mu_w \Bar{z})\in \mathbb{C} $$ 
where $\mu_w=\frac{i-w}{i+w}$, $\varphi_w$ is a Teichm\"uller mapping (a map with minimal dilatation) between $\mathbb{C}/(\mathbb{Z}+i\mathbb{Z})$ and $T_w=:\mathbb{C}/(\mathbb{Z}+w\mathbb{Z})$ (see \cite{imayoshi2012introduction}, \cite{nag} for more details). Hence, the metric : $$ h_w=\lambda(z)^{2}|dz+\mu_w d\Bar{z}|^2 $$ defines a new conformal structure on $\Sigma$ which is biholomorphic to $T_w$. 

 Fix $w_0 \in \mathbb{H}^{2}$, we want to construct a conformal embedding for the metric $h_{w_0}$. Let $B_\rho \subset \mathbb{H}^{2}$ be the closed disc centered at $w_0$ of radius $\rho$ and choose $\delta \in \mathbb{R}^{+}$ small enough so that the embedding $f$ is long for the metrics 
$g_w=\delta h_w$, meaning that 
$$f^*h-g_w$$
is positive definite for all $w\in B_\rho$. This is possible since $B_\rho$ is compact.

 Consider the map 
$$ \Delta : (z,w) \in C \times B_\rho \mapsto (f^*h-g_w)(z)\in S^{+}(\mathbb{R}^{2})$$
By construction, this map is continuous (as $g_w$ varies continuously with respect to $w$) and since $C\times B_\rho$ is compact, we can find $\ell_1\otimes \ell_1 , \ell_2\otimes \ell_2, ..., \ell_p\otimes \ell_p$ squares of linear forms such that 
$$ \Delta = \sum_{i=1}^{p} \eta_i \ell_i \otimes \ell_i$$ 
where each $\eta_i : C\times B_\rho \rightarrow \mathbb{R}^{+} $ is a smooth function. Moreover, since rational directions in $\mathbb{R}^{2}$ are dense and the image by $\Delta$ of $C\times B_\rho$ is compact, we can choose $\ell_1, \ell_2, ..., \ell_p$ so that their kernels contain non trivial vectors with integer coordinates, this will assure that the maps defined below are well defined on the quotient $\Sigma$.

 Using the decomposition above, we can build by iterating the corrugation process, a sequence of intermediary maps $f=F_{0}, F_1, ..., F_p=:F$ defined on $\Sigma \times B_\rho$, such that for each $w\in B_\rho$, we have:
$$F_i(.,w)=CP(F_{i-1}(.,w),\mu_i(.,w),\ell_i,N_i)$$
where $\mu_i = F_{i-1}^{*}h - \eta_i \ell_i \otimes \ell_i$. We note for simplicity, $F_p(.,w)=:F_w$.

Notice that the linear forms $\ell_1, \ell_2, ..., \ell_p$ are chosen uniformly for all $w\in B_\rho$. In addition, the coefficients $\eta_i$ vary continuously with $w$, and the corrugation numbers can be chosen uniformly, since $B_\rho$ is compact. Therefore, the map
 $$w\in B_\rho \longmapsto F_w^{*}h \in Met(\Sigma)$$ is continuous. Moreover, by construction, we have:
 $$\left\|F_w^{*}h-g_w\right\|_{C^{0}}=O\left(\frac{1}{N_1}\right)+...+ O\left(\frac{1}{N_p}\right)$$

For each $w\in B_\rho$, the metric $F_w^{*}h$ induces a conformal structure on $\Sigma$ that we will denote by $X_w$, and hence we have a map
$$G: w\in B_\rho \longmapsto (X_w, id_{\Sigma})\in \H^{2} $$
It follows from our construction that:
\begin{lemma} We have

     \begin{itemize}
        \item The map $G : B_\rho \rightarrow \mathbb{H}^{2}$ is continuous. 
        \item For any $\varepsilon>0$, we can choose the $N_1,...,N_p$ such that $d_{Teich}((X_w, id_{\Sigma}), (T_w, id_{\Sigma}))\leq \varepsilon$.
    \end{itemize}
\end{lemma}
\begin{proof}
    The first point is clear, since the following maps:
    $$  w\in B_\rho \longmapsto  F_w^{*}h \in Met(\Sigma)$$
    $$g \in Met(\Sigma) \longmapsto (X(g), id_\Sigma) $$
    are continuous. 

    To prove the second point, we only need to compute the dilatation $K$ of the identity map between $T_w$ and $X_w$. Recall that the conformal structure $T_w$ (respectively $X_w$) is induced by the metric $g_w$ (respectively $F_w^{*}h$). Therefore, we have 
    $$K^{2} = \sup_{z\in \Sigma}\left(\sup\frac{F_w^{*}h}{g_w}/\inf \frac{F_w^{*}h}{g_w}\right) = 1 + O\left(\frac{1}{N_1}\right)+...+ O\left(\frac{1}{N_p}\right) $$ for all $w\in B_\rho$ and where the supremum and infimum are taken over all directions. It suffice to choose the $N_i$'s large enough to conclude.
\end{proof}

If we take $\varepsilon=\rho$ in the lemma above, then, by proposition \ref{B}, we conclude that there exists $w_1 \in B_\rho $ such that $G(w_1)=w_0$. Therefore, the embedding $F_{w_1}$ realises a conformal model of $w_0$.

 \subsection{Higher genus case}
As before, we start with a spacelike embedding $f : \Sigma \rightarrow M$ of a closed smooth surface $\Sigma$ of genus $d \geqslant 2 $ in $M$. We will need, as in the genus one case, a local parametrization of the Teichm\"uller space by smooth Riemannian metrics defined on $\Sigma$. This will be given by the work of Fischer and Tromba \cite{fischer1984purely}. Denote by 
\begin{itemize}
    \item $\mathcal{M}$ the set of Riemannian metrics on $\Sigma$.
    \item $\mathcal{M_{-}}$ the set of Riemannian metrics of curvature $-1$ on $\Sigma$.
    \item $\mathcal{D}$ the group of orientation preserving diffeomorphisms of $\Sigma$.
    \item  $\mathcal{D}_\circ$ the connected component of the identity in $\mathcal{D}$.
    \item $\mathcal{P}$ the group of smooth positive functions $\lambda: \Sigma \rightarrow \R_{>0}$ on $\Sigma$.
    
\end{itemize}
We consider the following actions 
\begin{itemize}
    \item  Pointwise conformally equivalent metrics: 
    $$ \mathcal{P} \times \mathcal{M } \rightarrow \mathcal{M}$$
    $$(\lambda, g) \mapsto \lambda . g $$
    \item Conformally equivalent metrics up to isotopy: 
    $$ \mathcal{M}/\mathcal{P} \times \mathcal{D}_\circ \rightarrow \mathcal{M}/\mathcal{P}$$
    $$ ([g], \varphi ) \mapsto [\varphi^*g]$$
\end{itemize}
Now consider the projection 
$$\pi: \mathcal{M} \rightarrow \mathcal{M}/\mathcal{P}$$
and its restriction 
$$\pi_{-} : \mathcal{M}_{-} \rightarrow \mathcal{M}/\mathcal{P}$$

By the unformization theorem, any Riemannian metric on $\Sigma$ is pointwise conformally equivalent to a unique metric of curvature $-1$. In fact, the maps $\pi_{-}$ is a smooth diffeomorphism (see \cite{fischer1984purely}). 

We will consider now a new action that will give us a metric description of the Teichm\"uller space:
    $$ \mathcal{M}_{-} \times \mathcal{D}_\circ \rightarrow \mathcal{M}_{-}$$
    $$ (g, \varphi ) \mapsto \varphi^*g$$
    
    The Teichm\"uller space $Teich(\Sigma)$ is homeomorphic to the quotient $\mathcal{M}_{-}/\mathcal{D}_\circ$ and therefore, it is homeomorphic to $\frac{\mathcal{M}/\mathcal{P}}{\mathcal{D}_\circ}$. In particular, we get that $\mathcal{M}_{-}/\mathcal{D}_\circ$ is homeomorphic to $\mathbb{R}^{6d-6}$.

    We denote 
    $$\widehat{\pi}: \mathcal{M} \rightarrow \frac{\mathcal{M}/\mathcal{P}}{\mathcal{D}_\circ} \cong \mathcal{M}_{-}/\mathcal{D}_\circ $$
The last ingredient that we need is the following slice theorem:  
\begin{theorem}[\cite{fischer1984purely}]
    For any $g\in\mathcal{M}_{-1}$, there exists a local cross section $\mathcal{Z}_g \subset \mathcal{M}_{-}$ at $g$ for the action of $\mathcal{D}_\circ$. In other words,  there exists a smooth submanifold $\mathcal{Z}_g$ of $\mathcal{M}_{-1}$ containing $g$ such that the restriction:
    $$\widehat{\pi}: \mathcal{Z}_g \rightarrow \mathcal{M}_{-}/\mathcal{D}_\circ $$ is a homeomorphism. 
\end{theorem}
Hence, we have a local parametrization of the Teichm\"uller space by smooth Riemannian metrics. We are now ready to prove Theorem \ref{T1}: 
\medskip

Let $\widehat{\pi}(g_0)$ be a point in $Teich(\Sigma)$, where $g_0 \in \mathcal{M}_{-}$, and let $B_\rho$ be a closed ball centered at $\widehat{\pi}(g_0)$ and of radius $\rho$ such that $\widehat{\pi}^{-1}(B_\rho)\subset \mathcal{Z}_{g_0}$ (for simplicity, we identify $B_\rho$ and $\widehat{\pi}^{-1}(B_\rho)$). We assume that the embedding $f : \Sigma \rightarrow M$ is long with respect to each metric in $B_\rho$. We can always find $\lambda \in \mathbb{R}^{+} $ such that $\lambda f^*h - g $ is  positive definite for all the metrics in $B_\rho$, since both $B_\rho$ and $\Sigma$ are compact.

 We consider a finite set of local parametrizations of $\Sigma$ ($\phi_i: U_i \subset \mathbb{R}^2 \rightarrow V_i \subset \Sigma $) and a partition of unity $\{\rho_i\}$ subordinated to it ($supp(\rho_i \subset V_i$) for all $i\in [\![1,k]\!]$). We define the parametric isometric default: 
$$\Delta(z,g) = (f^*h-g)(z)$$
We note $\Delta_i=: \phi_i^*\Delta|_{supp(\rho_i)}$. By continuity of $\Delta_i$ and compactness of $supp(\rho_i)$ and $B_\rho$,  we can find $\ell_{i,1}\otimes \ell_{i,1} , \ell_{i,2}\otimes \ell_{i,2}, ..., \ell_{i,p_i}\otimes \ell_{i,p_i}$ squares of linear forms such that 
$$\Delta_i = \sum_{j=1}^{{p_i}} \eta_{i,j} \ell_{i,j} \otimes \ell_{i,j}$$ 
where $\eta_{i,j} : supp(\rho_i) \times B_\rho \rightarrow \mathbb{R}^{+} $ is a smooth function. We can use the construction presented in section $2$ to build for each $g\in B_\rho$ a spacelike embedding $F_g : \Sigma \rightarrow M$ such that 
   \begin{equation}\label{equa}
       \left\|F_g^{*}h-g\right\|_{C^{0}}\leq \epsilon
   \end{equation}
for any given $\epsilon>0$. Moreover, since each coefficient function $\eta$ vary continuously on $B_\rho$, the family of loops used in the corrugation process depends continuously on $g$. Thus, the map 
$g\in B_\rho \mapsto F_g^*h \in \mathcal{M}$
is continuous.

The choice of the numbers of corrugations $N_{i,j}$'s can be done uniformly as we are working on compact sets $B_\rho$, $\Sigma$. 

Let $G(g):=\widehat{\pi}({F_g}^*h) \in Teich(\Sigma) $. As in the genus one case, it follows immediately from our construction:
\begin{lemma} We have

     \begin{itemize}
        \item The map $G : B_\rho \rightarrow Teich(\Sigma)$ is continuous. 
        \item  For any $\varepsilon>0$, $d_ {Teich}(G(g) - \widehat{\pi}(g)) \leq \varepsilon$.
    \end{itemize}
\end{lemma}
\begin{proof}
    As before, by construction, the following maps:
    $$  g\in B_\rho \longmapsto  F_w^{*}h \in \mathcal{M}$$
    $$g \in \mathcal{M} \longmapsto \widehat{\pi}(g) \in \frac{\mathcal{M}/\mathcal{P}}{\mathcal{D}_\circ} \cong Teich(\Sigma) $$
    are continuous. 

    Using inequality \ref{equa}, the dilatation of the identity map between the Riemann surface structure defined by $g$ and the one defined by $F_g^{*}h$ is: 
    $$K^{2} = \sup_{z\in \Sigma}\left(\sup\frac{F_g^{*}h}{g}/\inf \frac{F_g^{*}h}{g} \right)\leq 1 + \epsilon$$ 
    for all $g\in B_\rho$ and where the supremum and infimum are taken over all directions. Choosing $\epsilon$ small enough, we conclude by Remark \ref{dila}.
\end{proof}

Choosing $\varepsilon=\rho$ in the above lemma, we conclude by proposition \ref{B} that there exists $g_1 \in B_\rho $ such that $G(g_1)=\widehat{\pi}(g_0)$. Therefore, the embedding $F_{g_{1}}$ realizes a conformal model of $\widehat{\pi}(g_0)$. This conclude the proof of Theorem \ref{T1}.
  \section{Proof of Theorem \ref{T2}}
In this section we present the proof of Theorem \ref{T2} and Corollary \ref{C2}. We recall before that some notations:

 Let $\mathcal{T}$ be the open solid timelike cone in the Minkowski space of dimension $2+1$ .i.e
$$\mathcal{T}= \{ (x,y,z)\in \mathbb{R}^{2,1} |  \; x^{2}+y^{2} < z^{2} \; ; \; z> 0 \}$$
Remember that $\mathcal{T}$ with the induced Lorentzian metric is isometric to $\mathbb{H}^{2}\times \mathbb{R}_{>0}$ with the metric $r^{2}h-dr^{2}$, where $h$ is the hyperbolic metric on $\mathbb{H}^{2}$. For simplicity, we denote by $\mathbb{H}^{2}_r$ the level $\mathbb{H}^{2}\times \{r\}$.

 Let $\Gamma$ be a cocompact lattice of the group $SO^{\circ}(2,1)$. Since $\Gamma$ acts by isometries on $\mathcal{T}$ and preserves the levels $\mathbb{H}^{2}_r$, the quotient $\mathcal{T}/ \Gamma $ is a warped product $(S\times \mathbb{R}_{>0}, r^{2}h_\Gamma-dr^{2})$, where $S=\mathbb{H}^{2}_1/\Gamma$ is a closed hyperbolic surface. We will note by $\Sigma$ the supporting topological surface of genus $d\ge 2$ and by $S_r$ the level $S=\mathbb{H}^{2}_r/\Gamma$.
\\ Following the notations of the previous section, we consider:
\begin{itemize}
    \item $\mathcal{M}_{<0}$ the set of Riemannian metrics on $\Sigma$ of negative curvature.
    \item $\widetilde{B}(\Sigma)$ the set of metrics in $\mathcal{M}_{<0}$ that can be isometrically embedded in $\mathcal{T}/\Gamma$.
    \item $B(\Sigma)$ the projection of $\widetilde{B}(\Sigma)$ in $Teich(\Sigma)$.
\end{itemize}
Recall that we want to prove the following:
\begin{theorem*}\label{T2*}
    Any $g\in \widetilde{B}(\Sigma)$ is up to scaling, quasi-isometric to the hyperbolic metric $h_\Gamma$ with universal coefficients. More precisely, there exists $C>1$ that depend only on $\Gamma$, such that : 
    $$\frac{1}{C} h_\Gamma \le g \le C h_\Gamma$$. 
\end{theorem*}
\begin{corollary*}
    The set $B(\Sigma)$ is relatively compact. 
\end{corollary*}

\begin{proof}[Proof of Corollary]
    Let $g\in \widetilde{B}(\Sigma)$ and denote by $X(h_\Gamma)$ and $X(g)$ the complex structure on $\Sigma$ induced by $h_\Gamma$ and $g$ respectively. By Remark \ref{dila}, the dilatation of the identity map of $\Sigma$ seen as quasiconformal map between $X(h_\Gamma)$ and $X(g)$ is given by:  $$ K=\sup_{z\in\Sigma}\sqrt{\sup\frac{g}{h_\Gamma}\bigg/\inf \frac{g}{h_\Gamma}}$$
 where the supremum and infimum are taken over all directions. By theorem \ref{T2} above, we have :
 $$K^{2} \leq C^{2}$$
 In particular, the Teichm\"uller distance:
  $$d_{Teich}((X(g),id),(X(h_\Gamma),id))\le \frac{1}{2} \log(C^{2})$$ 
  We conclude that $B(\Sigma)$ is included in the ball of center $(X(h_\Gamma),id)$ and of radius $\frac{1}{2} \log(C^{2})$. In particular, $B(\Sigma)$ is relatively compact.
\end{proof}
 
 In order to prove theorem \ref{T2}, we first reformulate the problem into an equivariant one. Let $g\in \widetilde{B}(\Sigma)$, and $F:(\Sigma,g) \rightarrow \mathcal{T}/\Gamma$ be an isometric embedding of $(\Sigma,g)$. Lifting the metric $g$ into a $\Gamma$-invariant metric $\widetilde{g}$ on the universal cover $\widetilde{\Sigma}$, this is equivalent to the existence of a $\Gamma$-equivariant isometric embedding $$\widetilde{F} : (\widetilde{\Sigma}, \widetilde{g}) \rightarrow \mathcal{T}$$
with respect to the group isomorphism $\varrho : \pi_1(\Sigma) \rightarrow \Gamma$, this means that $$\widetilde{F}(\gamma . x) = \varrho(\gamma) . \widetilde{F}(x).$$

 We will now prove some facts about this embedding:
 
  A theorem of Mess \cite{mess2007lorentz} asserts that if the induced metric by a spacelike immersion in the Minkowski space is complete, then the image of this immersion is a spacelike entire graph, i.e. it is of the form $\{(x,y,u(x,y))\; \; x,y\in \R \} $ for some function $u:\mathbb{R}^{2} \rightarrow \mathbb{R}$. In our case, the induced metric $\widetilde{g}$ is $\Gamma$-invariant, and since $\Gamma$ is cocompact, this implies completeness. Therefore, $\widetilde{F}(\widetilde{\Sigma})$ is the graph in $\R^{2,1}$ of some smooth function that we call $u$.
  
  We have the following formula of the curvature:
$$\kappa(\widetilde{g}) = - \frac{\partial_{xx}u \partial_{yy}u - \partial_{xy}u^{2}}{(1-\partial_x u^{2}-\partial_y u^{2})^{2}}$$
and since $\kappa(\widetilde{g})$ is negative, this implies that $u$ is strictly convex; this is why we call $\widetilde{F}$ a convex embedding. Since we want to compare the metrics $g$ and $h_\Gamma$, we will need the following lemma: 

\begin{lemma}
   The image of $\widetilde{F}$ in $\mathbb{H}^{2}\times \mathbb{R}_{>0}$ is a graph of some function over $\H^{2}$.
\end{lemma}
\begin{proof}
This is equivalent to the fact that $F$ is a graph of some function over $S$ in $S\times \R_{>0}$. 

Since the kernel of the differential of the projection $pr_1: S\times \R_{>0} \rightarrow S $ is timelike and $F$ is a spacelike embedding, we get by the inverse function theorem that $pr_1\circ F$ is a local diffeomorphism. In fact, $pr_1\circ F$ is a covering map (since $\Sigma$ is compact and connected). We conclude that $pr_1\circ F$ is a diffeomorphism, since $\Sigma$ has genus $\geq 2$. 
\end{proof}

  We are now ready to prove Theorem \ref{T2}, we recall that the goal is to find $1< C$, such that for any $g\in \widetilde{B}(\Sigma)$, we have 
  $$\frac{1}{C} h_\Gamma \le g \le C h_\Gamma.$$
  
  The main idea here is to prove that that there exists a number $\alpha$ that depends only on $\Gamma$, such that the image of any convex embedding is trapped between two levels $\H^{2}_r$ and $\H^{2}_{\alpha r}$, equivalently in $\mathcal{T}/\Gamma$, the image of any convex embedding is bounded by two levels $S_r:=\H^{2}_r/\Gamma$ and $S_\alpha:=\H^{2}_{\alpha r}/\Gamma$. More precisely, we have the following proposition: 
 \begin{proposition}\label{pr1}
     There exists $\alpha>1$ that depends only on $\Gamma$, such that, for any $g \in \widetilde{B}(\Sigma)$, up to rescaling, $F(\Sigma)$ is bounded by $S_\alpha$ and $S_1$. In particular, we have for $C=\alpha^{2}$:
    $$g \le C h_\Gamma.$$
 \end{proposition}

\begin{proof}
    The inequality is a straightforward consequence. As proven above, the image of the embedding $\widetilde{F}$ is a graph over $\H^{2}$:
$$ \{(x,y,f(x,y) | \;\; (x,y)\in \H^{2} \;\; \text{and } \;\; f:\H^{2}\rightarrow \R_{>0} \} $$
If such an $\alpha$ exists, then up to scaling: $f(\H^{2})\subset [1,\alpha]$. Hence, the induced metric by $\widetilde{F}$ (up to identifying $\mathcal{T}$ with $\H^{2}\times \R_{>0}$) satisfies: 
$$f^{2}h-df^{2} \leq \alpha^{2} h$$
where $h$ is the hyperbolic metric on $\H^{2}$.

 By equivariance, we need to prove that the image of $\widetilde{F}$ is in the past of $\mathbb{H}^{2}_\alpha$ and in the future of $\mathbb{H}^{2}_1$.

Up to rescaling, we can always assume that $\widetilde{F}(\widetilde{\Sigma})$ is in the future of $\mathbb{H}^{2}_1$ and is tangent to it at some point $p$ and therefore at all the points of the orbit $\Gamma . p$. Consider now the closed convex hull of $\Gamma . p$ denoted by $\mathcal{H}_p$ and let $\partial \mathcal{H}_p$ be its boundary. Since $\Gamma$ acts by linear transformations and $\Gamma . p$ is $\Gamma$-invariant, $\mathcal{H}_p$ is $\Gamma$-invariant and so is $\partial \mathcal{H}_p$.

 The level function $\mathcal{T}\cong \H^{2}\times \R_{>0} \rightarrow \R_{>0} $ is a continuous function. We want to prove that its restriction on $\widetilde{F}(\widetilde{\Sigma})$ is bounded. Since $\widetilde{F}(\widetilde{\Sigma})$ contains $\Gamma.p$ and it is the graph of a convex function $u$, it is in the past of $\partial \mathcal{H}_p$. Therefore, it suffice of prove that the restriction of the level function on $\partial \mathcal{H}_p$ is bounded. 

Consider the Klein model of the hyperbolic plane modeled on the disc $$D=\{(x,y,1) \; \; x^{2}+y^{2} < 1\}$$ Recall that we get such a model by scaling each point of the hyperboloid $\H^{2}_1$ to be in $D$. Consider now the map $\varphi : \mathcal{T} \rightarrow D$ that associates to each point $q\in \mathcal{T}$, the point $\varphi(q)$ defined as the intersection of $D$ and the line $L_q$ passing through $q$ and the origin. It is readily seen that $$ \forall (x,y,z)\in \mathcal{T} \;\;\;\; \varphi (x,y,z)=(\frac{x}{z},\frac{y}{z},1)$$ 
Euclidean segments are mapped by $\varphi$ to Euclidean segments in $D$ (or to a point if the segment lies inside a line passing through the origin).

We consider now the maps $\varphi_0$ (respectively $\varphi_1$) restriction of $\varphi$ on $\partial \mathcal{H}_p$ (respectively on $\H^{2}_1$). The map $\varphi_1$ is an isometry between the hyperboloid model and Klein's model. We will admit all the results concerning this isometry (see \cite{iversen1992hyperbolic} or \cite{thurston1997three} for more details). We admit for now that the map $\varphi_0$ is a homeomorphism (we will prove this fact below, see lemma \ref{lm}).

The map $\varphi_0$ is $\Gamma$-equivariant, and since the action of $\Gamma$ on $D$ is cocompact, hence its action on $\partial \mathcal{H}_p$ is cocompact. Restricting the level function $\mathcal{T}\cong \H^{2}\times \R_{>0} \rightarrow \R_{>0} $ on $\partial \mathcal{H}_p$, we find by continuity, $\alpha_p > 1$ such that $\partial \mathcal{H}_p$ is in the past of $\H^{2}_{\alpha_p}$. Let's take $\alpha_p$ to be minimal for this property. 

 We need now to prove that we can find a uniform $\alpha$ such that $\alpha_p \le \alpha $ for any $p\in \H^{2}_1$. Fix a fundamental polygon $P\subset \H^{2}_1$ for the action of $\Gamma$ that contains $(0,0,1)$ and let $\widehat{P}$ be its image in $D$ under $\varphi_1$. Since $\Gamma$ is cocompact, both $P$ and $\widehat{P}$ are compact. Moreover, we have that for any $p\in \widehat{P}$ and any $q\in \widehat{P}$, we can find $\gamma_1,\gamma_2,\gamma_3 \in \Gamma$, such that $q$ lies in the interior of the Euclidean triangle with vertices $\gamma_i . p$, and moreover the vertices can be taken in the closed ball for the hyperbolic metric $B\subset D$ of center $0$ and radius $n . diam(P)$ for some $n\geq 1$. This implies by definition of $\varphi_0$ that $\varphi_0^{-1}(P)\subset Conv(\varphi_1^{-1}(B))$, where $Conv(\varphi_1^{-1}(B))$ is the convex hull of $\varphi_1^{-1}(B)$. By Caratheodory's theorem, $Conv(\varphi_1^{-1}(B))$ is compact since $\varphi_1^{-1}(B)$ is. Restricting the level function to $Conv(\varphi_1^{-1}(B))$, we find the uniform bound $\alpha$.
\end{proof}
\begin{lemma}\label{lm}
    The map $\varphi_0$ is a homeomorphism.
\end{lemma}
\begin{proof}
    To prove that $\varphi_0$ is bijective, we need to prove that for any point $q\in D$, the half line $L_q$ starting at the origin and passing through $q$ intersects $\partial \mathcal{H}_p$ exactly one time. In order to do this, we need to consider the image $\varphi_1(\Gamma . p) \subset D$. Since $\Gamma$ is cocompact, the limit set of $\Gamma$ is the boundary of $\H^{2}_1$ (see \cite{katok1992fuchsian}) and hence, the boundary of the Klein disc $\partial D = \{(x,y,1) \; \; x^{2}+y^{2} = 1\}$. Therefore, for any point $q\in D$, there exists $\gamma_1, \gamma_2, \gamma_3 \in \Gamma$ such that $q$ is in the interior of the euclidean triangle with vertices $\varphi_1(\gamma_1 . p), \varphi_1(\gamma_2 . p), \varphi_1(\gamma_3 . p)$. This implies that the half line $L_q$ intersects the interior of the euclidean triangle  with vertices $\gamma_1 . p, \gamma_2 . p, \gamma_3 . p$. and hence it intersects $\mathcal{H}_p$. Applying the same argument for a triangle with vertices closer to the boundary (this is always possible since the limit set of $\Gamma$ is the whole boundary $\partial D$) gives us a sequence of points in $L_q \cap \mathcal{H}_p$ converging to infinity (in the half line $L_q$). Since both $L_q$ and $\mathcal{H}_q$ are convex, this implies that as soon as a point $q_1$ is in $L_q\cap \mathcal{H}_p$, the half line starting from $q_1$ with the same direction as $L_q$ is included in $L_q\cap \mathcal{H}_p$. Consider now for the natural order on $L_q$
    $$\widetilde{q}=inf \{q_1\in L_q\cap \mathcal{H}_p \}$$
    Since both $L_q$ and $\mathcal{H}_p$ are closed, this will imply that $\widetilde{q}\in L_q$ and since $\widetilde{q}$ is not in the interior of $L_q\cap \mathcal{H}_p$, it is not in the interior of $\mathcal{H}_p$. Hence, $\widetilde{q}$ lies in the boundary $\partial \mathcal{H}_p$, and therefore, $\varphi_0$ is surjective. Moreover, $\widetilde{q}$ is the only point in $L_q \cap \partial \mathcal{H}_p$, as for any different point $q_1 \in L_q\cap \mathcal{H}_p$, we can find $\gamma_1, \gamma_2, \gamma_3 \in \Gamma$, such that $q_1$ is in the interior of the tetrahedron with vertices $\widetilde{q}, \gamma_1 . p,\gamma_2 . p,\gamma_3 . p$, which implies that $q_1$ is in the interior of $\mathcal{H}_p$. Therefore, $\varphi_0$ is also injective. 
    
     The map $\varphi_0$ is continuous since it is a restriction on a closed set of $\varphi$ which is clearly continuous. We will prove now that $\varphi_0$ is proper: Fix $K\subset D$ compact, hence it is contained in the interior of an Euclidean polygon $P \subset D$ with vertices $q_1,\ldots,q_n$ in $\varphi(\Gamma .p)$, Consider $\widetilde{P}$ to be the closed convex hull of $\varphi_1^{-1}(q_1),\ldots,\varphi_1^{-1}(q_n)$ and $Q$ to be $\varphi_1^{-1}(P)$ which is compact by continuity of $\varphi_1^{-1}$. Any point of $q\in K$ lies in a Euclidean triangle $\Delta$ with vertices $q_{i_1},q_{i_2},q_{i_3}\in \{q_1, \ldots , q_n \}$, and the half line $L_q$ starting from the origin and containing $q$ must intersect $Q\subset \H^{2}_1$ before intersecting $\mathcal{H}_p$, since by convexity $\mathcal{H}_p$ is in the future of $\H^{2}_1$. This implies that $\varphi_0^{-1}(q)$ lies inside the tetrahedron with vertices $\varphi_1^{-1}(q), \varphi_1^{-1}(q_{i_1}),\varphi_1^{-1}(q_{i_2}),\varphi_1^{-1}(q_{i_3})$. Hence $\varphi_0^{-1}(K)\subset C_P$, where $C_P$ is the closed convex hull of $\varphi_1^{-1}(P)$ which is compact, since $P$ is. Therefore, $\varphi_0^{-1}(K)$ is compact.
\end{proof}
 We can proceed now to find the lower bound in the theorem: 
 \begin{proposition}\label{pr2}
     There exists $C^\prime < 1 $ that depends only on $\Gamma$, such that the induced metric by $\widetilde{F}$ seen as a graph of a function $f$ over $\H^{2}$ satisfies $C^\prime h\leq f^{2}h-df^{2} $.
 \end{proposition}

\begin{proof}

Let $p=(0,0,r)$, the tangent plane $T_p\H^{2}_r$ is parallel to the $xy$-plane and the normal line $N_p\H^{2}_r$ is parallel to the $z$-axis. We can use proposition \ref{pr1} to restrict $1\leq r \leq \alpha$ (since $Im(f)\subset[1,\alpha]$). The statement of the proposition at the point $p$ can be formulated as follows: 
$$\text{ For some $0<c<1$,} \;\;\; T_p\widetilde{F}(\widetilde{\Sigma}) \text{ lies outside the cone of equation } c(x^{2}+y^{2})-(z-r)^{2}=0. $$
This means that points of $T_p\widetilde{F}(\widetilde{\Sigma})$ verify the inequality 
$$c(x^{2}+y^{2})-(z-r)^{2} > 0.$$

 Now, since $\widetilde{F}(\widetilde{\Sigma}) $ is the graph of a convex function $u$ over $\R^{2}$, we have that the hypergraph of $u$ (convex hull of the graph of $u$) lies entirely above the tangent plane at any point of $\widetilde{F}(\widetilde{\Sigma}) $. Indeed, using the convexity of $u$ (see \cite{boyd2004convex}), If $X\in \R^{2} $, then we have
$$ \forall \;Y \in \R^{2} \;\;\; u(Y) \ge u(X) + \langle grad(u), (Y-X)\rangle $$
where $\langle .,.\rangle$ is the usual scalar product of $\E^{2}$ and $grad(u)$ is the gradient of the function $u$ evaluated at the point $X$. Hence, we have: 
$$\left< \begin{pmatrix}
    -grad(u) \\ 1
\end{pmatrix}, \begin{pmatrix}
    Y-X \\ u(Y)-u(X)
\end{pmatrix} \right> \ge 0$$

Proposition \ref{pr1} above states that the embedding $\widetilde{F}(\widetilde{\Sigma}) $ is in the past of $\H^{2}_\alpha$. In other words, the hypergraph of $u$ contains $\H^{2}_\alpha$. In particular, we have for any $(Y,z)\in \H^{2}_\alpha$
$$u(Y)< z$$
which implies that for any $X\in \R^{2}$ 
\begin{equation}\label{eq}
    \left<\begin{pmatrix}
    -grad(u) \\ 1
\end{pmatrix}, \begin{pmatrix}
    Y-X \\ z-u(X)
\end{pmatrix} \right > > 0
\end{equation}
the inequality is strict by taking $\alpha$ large enough. 
We will prove now that if the tangent plane intersects (transversely) $\H^{2}_\alpha$, then the previous inequality can not hold. 

We will focus on the point $p=(0,0,r)$, take a plane $P$ passing through $p$ and intersecting $\H^{2}_\alpha$. Since $\H^{2}_\alpha$ is fixed under rotations around the $z$-axis, we can assume that there is an intersection point that lies on the plane $y=0$. Taking the intersection of $P$ and $\H^{2}_\alpha$ with the plane $y=0$, we can reduce the work in the plane $y=0$, (see Figure \ref{fig1})
\renewcommand{\thefigure}{1}
\begin{figure}[h!]
\includegraphics[width=8cm]{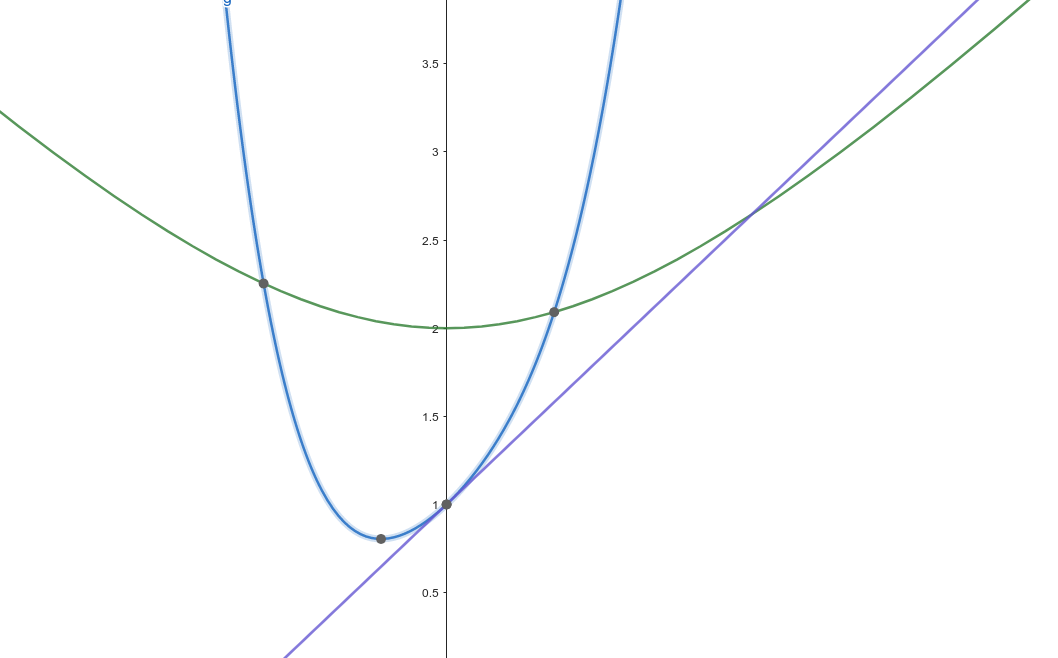}
\centering
\caption{Intersection of the tangent plane (in purple) with the level $\H^{2}_\alpha$ (in green) implies that the graph of $u$ (in blue) is not bounded by $\H^{2}_\alpha$}
\label{fig1}
\end{figure}

Let $l:=P\cap \{y=0\}$ to be a line in the $xz$-plane parametrized by $(t,at+r)$ ($a^{2}<1$ since $P$ is a spacelike plane) passing through $(0,r)$ and consider the curve $\mathcal{C}:=\H^{2}_\alpha \cap \{y=0\}$ (whose equation is $x^{2}-z^{2}=-\alpha^{2}$ with $z>0$). To study the intersection $l\cap \mathcal{C}$, we need to solve the equation 
$$ t^{2} - (at+r)^{2} = - \alpha^{2}$$
This quadratic equation has two distinct solutions for any
$$a^{2}> 1-  \frac{r^{2}}{\alpha^{2}}$$
It is clear now that if 
$$1-  \frac{1}{\alpha^{2}} < a^{2} < 1$$
then we have for any $1\leq r \leq \alpha $, for $t=\frac{ar}{1-a^{2}}$ :
$$t^{2} - (at+r)^{2} + \alpha^{2}  < 0 $$
which geometrically means that $l$ lies above $C$ at this point (see figure \ref{fig1}). Coming back to our problem, the previous fact implies that for equation (\ref{eq})  to hold, $l$ must lie outside the cone of equation $$(1-\frac{1}{\alpha^{2}})x^{2}-(z-r)^{2}=0$$
which equivalently means that the tangent plane $T_p\widetilde{F}(\widetilde{\Sigma}) $ must lie outside the cone of equation $$(1-c(x^{2}+y^{2})-(z-r)^{2}=0$$
Where $c=1-\frac{1}{\alpha^{2}}$, which implies that 
$$C^\prime h_p\leq f^{2}(p)h_p-d_pf^{2} $$
for some $C^\prime$ that depends on $c$.

This generalizes directly to any point. Indeed, for any point $p\in\widetilde{F}(\widetilde{\Sigma}) $ we can find an isometry $\gamma \in SO^{\circ}(2,1)$ that sends $p$ to $(0,0,r)$ for $1\leq r \leq \alpha$. Such isometry preserves the levels $\H^{2}_t$ and sends $T_p\H^{2}_r$ to $T_{(0,0,r)}\H^{2}_r$ and the normal $N_p\H^{2}_r$ to $N_{(0,0,r)}\H^{2}_r$ and clearly preserves the intersection if existed. 

\end{proof}
Combining propositions \ref{pr1} and \ref{pr2} we have Theorem \ref{T2}.
\bibliographystyle{plain}
\bibliography{main.bib}

\end{document}